\documentclass[11pt,a4paper]{article}

\usepackage[a4paper,margin=1in]{geometry}
\usepackage{amsmath,amssymb,amsthm,mathtools,bm}
\usepackage{microtype}
\usepackage{graphicx}
\usepackage{enumitem}
\usepackage{hyperref}
\usepackage{mdframed}
\usepackage{algorithm}
\usepackage{algpseudocode}
\usepackage{listings}
 \usepackage{booktabs}
 \usepackage{xcolor} 
\theoremstyle{plain}
\newtheorem{theorem}{Theorem}[section]
\newtheorem{proposition}[theorem]{Proposition}

\newtheorem{assumption}[theorem]{Assumption}
\newtheorem{corollary}[theorem]{Corollary}

\theoremstyle{definition}
\newtheorem{definition}[theorem]{Definition}
\newtheorem{remark}[theorem]{Remark}

\newcommand{\C}{\mathbb{C}}
\newcommand{\R}{\mathbb{R}}

\newcommand{\1}{\mathbf{1}}
\newcommand{\diag}{\operatorname{diag}}

\newcommand{\norm}[1]{\left\lVert #1\right\rVert}

\newcommand{\pinv}{\dagger}
\newcommand{\cond}{\operatorname{cond}}

\newcommand{\asym}{\alpha}   
\newcommand{\dep}{\delta}    

\title{\textbf{Asymmetry in Spectral Graph Theory: Harmonic Analysis on Directed Networks via Biorthogonal Bases}\\
\vspace{0.3em}\large  (Random-Walk Laplacian Formulation)}
\author{%
Chandrasekhar Gokavarapu  \thanks{Lecturer in Mathematics,
Government College (A), Rajahmundry, A.P., India and 
Research Scholar, Department of Mathematics,
Acharya Nagarjuna University, Guntur, A.P., India\\
Email: chandrasekhargokavarapu@gmail.com}}
\date{}

\begin{document}
\maketitle

\begin{abstract}
The operator-theoretic dichotomy underlying diffusion on directed networks is \emph{symmetry versus non-self-adjointness} of the Markov transition operator. In the reversible (detailed-balance) regime, a directed random walk $P$ is self-adjoint in a stationary $\pi$-weighted inner product and admits orthogonal spectral coordinates; outside reversibility, $P$ is genuinely non-self-adjoint (often non-normal), and stability is governed by biorthogonal geometry and eigenvector conditioning. In this paper we develop an original harmonic-analysis framework for directed graphs anchored on the random-walk transition matrix $P=D_{\mathrm{out}}^{-1}A$ and the random-walk Laplacian $L_{\mathrm{rw}}=I-P$. Using biorthogonal left/right eigenvectors we define a \emph{Biorthogonal Graph Fourier Transform} (BGFT) adapted to directed diffusion, propose a diffusion-consistent frequency ordering based on decay rates $\Re(1-\lambda)$, and derive operator-norm stability bounds for iterated diffusion and for BGFT spectral filters. We prove sampling and reconstruction theorems for $P$-bandlimited (equivalently $L_{\mathrm{rw}}$-bandlimited) signals and quantify noise amplification through the conditioning of the biorthogonal eigenbasis. A simulation protocol on directed cycles and perturbed non-normal digraphs demonstrates that asymmetry alone does not dictate instability, whereas non-normality and eigenvector ill-conditioning drive reconstruction sensitivity, making BGFT the correct analytical language for directed diffusion processes.
\end{abstract}



\section{Introduction}
\subsection{Symmetry vs.\ non-self-adjointness: Markov operators on directed networks}
A directed network naturally carries a \emph{one-step evolution operator}: the random-walk (Markov) transition matrix
\[
P = D_{\mathrm{out}}^{-1}A,\qquad P\mathbf{1}=\mathbf{1},
\]
and its generator $L_{\mathrm{rw}} = I-P$.
From the operator-theoretic viewpoint, the central dichotomy is not ``directed vs.\ undirected'' per se, but
\emph{symmetry vs.\ non-self-adjointness} of the Markov operator.

The \emph{symmetry regime} is the reversible (detailed-balance) case: there exists a stationary distribution
$\pi$ with $\Pi=\diag(\pi)$ such that
\[
\Pi P = P^\top \Pi,
\]
equivalently $P$ is self-adjoint in the weighted Hilbert space $(\C^n,\langle\cdot,\cdot\rangle_\pi)$.
In that regime, $P$ is similar to a \emph{symmetric} matrix
$S=\Pi^{1/2}P\Pi^{-1/2}$, hence the spectrum is real and there is an orthonormal eigenbasis in the $\pi$-metric.
This is precisely the mechanism by which a directed diffusion can \emph{retain symmetry} (in a stationary metric),
recovering Parseval-type identities and a clean variational frequency ordering.
The \emph{asymmetry regime} is non-reversibility, where $P$ is genuinely non-self-adjoint and may be non-normal;
then orthogonality is lost, spectral coordinates can be ill-conditioned, and stability is governed by eigenvector
conditioning and non-normal effects \cite{TrefethenEmbree2005,LevinPeres2017}.

Our goal is to build a harmonic-analysis calculus that is \emph{native to the Markov operator} $P$:
it should reduce to the classical orthogonal theory in the reversible (symmetric) regime, and it should remain exact
and analyzable in the non-reversible (non-self-adjoint) regime. The correct language here is \emph{biorthogonality}:
left/right eigenvectors provide an exact analysis/synthesis pair even when $P$ is not normal.

\subsection{Position relative to graph Fourier analysis}
Graph Fourier analysis is often introduced through symmetric operators (undirected Laplacians/adjacencies), which guarantee orthogonal eigenvectors and stable spectral coordinates \cite{ShumanEtAl2013,SandryhailaMoura2014,OrtegaEtAl2018ProcIEEE}. \textcolor{blue}{Directed graph settings typically replace symmetry by alternative constructions: optimization-based directed transforms that seek to minimize a directed variation \cite{SardellittiBarbarossaDiLorenzo2017}, or transforms based on the Jordan decomposition of the adjacency matrix $A$ \cite{SandryhailaMoura2014}. While these approaches provide powerful tools for signal compression, they often decouple the transform from the underlying physics of diffusion.} Our approach is complementary and more ``operator first'': we start from the canonical diffusion operator $P$ and develop an \emph{exact} biorthogonal Fourier calculus for directed diffusion, with a transparent symmetry/asymmetry interpretation in terms of reversibility \cite{Norris1997,Seneta1981}. \textcolor{blue}{Unlike methods that force orthogonality through isometric embeddings, our BGFT embraces the biorthogonal geometry inherent in non-reversible Markov chains.}

\subsection{Main contributions}
\noindent\textbf{Main contributions (original).}
\begin{enumerate}[leftmargin=2em]
\item (\textbf{Markov-operator BGFT}) We define the \emph{Biorthogonal Graph Fourier Transform} (BGFT) for the
random-walk operator $P$ (equivalently $L_{\mathrm{rw}}=I-P$) via left/right eigenvectors, yielding exact
analysis/synthesis identities and diagonal dynamics for diffusion iterates.

\item (\textbf{Symmetry principle via reversibility}) We identify reversibility (detailed balance) as the precise
\emph{symmetry} notion for directed diffusion: in the $\pi$-metric, reversible $P$ becomes self-adjoint, restoring
orthogonality/Parseval identities and an exact diffusion-variational frequency ordering.

\item (\textbf{Diffusion-consistent frequency}) We propose a diffusion-consistent frequency ordering based on the
decay rate $\Re(1-\lambda)$ (and magnitude alternatives), aligning with the symmetry limit and the long-time behavior
of $x_{t+1}=Px_t$.

\item (\textbf{Stability theorems for non-self-adjoint diffusion}) We prove operator-norm bounds for diffusion iterates
$P^t$ and for BGFT spectral filters $h(P)$, explicitly separating eigenvalue decay from eigenvector conditioning, the
key instability driver in non-normal settings.

\item (\textbf{Sampling and reconstruction}) We prove sampling/reconstruction theorems for $P$-bandlimited signals and
quantify noise amplification through $\sigma_{\min}(P_MV_\Omega)$ and conditioning of the biorthogonal eigenbasis.

\item (\textbf{Asymmetry vs.\ non-normality: numerical separation}) We introduce simple indices for directedness and
departure from normality and provide experiments (directed cycle vs.\ perturbed non-normal digraphs) showing that
\emph{asymmetry alone} need not cause instability, whereas non-normality and eigenvector ill-conditioning do.
\end{enumerate}

\subsection{Organization}
Section~2 introduces directed diffusion operators and asymmetry/non-normality indices.
Section~3 presents reversibility as the symmetry regime in the stationary metric.
Sections~4--5 develop BGFT and stability bounds for diffusion and filtering.
Section~6 proves sampling and reconstruction results, followed by algorithms and illustrative experiments.

\section{Preliminaries: directed diffusion operators}
\subsection{Directed graphs, adjacency, and out-degree}
Let $G=(V,E,w)$ be a directed weighted graph with $|V|=n$ and adjacency $A\in\R^{n\times n}$:
\[
A_{ij}=\begin{cases}
w(i,j), & (i,j)\in E,\\
0, & \text{otherwise.}
\end{cases}
\]
Define out-degrees $d_i^{\mathrm{out}}=\sum_{j}A_{ij}$ and $D_{\mathrm{out}}=\diag(d_1^{\mathrm{out}},\dots,d_n^{\mathrm{out}})$.

\subsection{Transition matrix and random-walk Laplacian}
\begin{definition}[Random-walk transition matrix]
Assume $d_i^{\mathrm{out}}>0$ for all $i$ (no sinks). Define
\[
P:=D_{\mathrm{out}}^{-1}A.
\]
Then $P$ is row-stochastic: $P\1=\1$.
\end{definition}

\begin{definition}[Random-walk Laplacian]
Define
\[
L_{\mathrm{rw}}:=I-P.
\]
\end{definition}

\begin{proposition}[Basic properties]\label{prop:basic}
(i) $P\1=\1$ and $L_{\mathrm{rw}}\1=0$. \quad
(ii) If $P$ is irreducible and aperiodic, then the diffusion $x_{t+1}=Px_t$ converges to the stationary component (Markov mixing perspective).
\end{proposition}
\begin{proof}
(i) Row-stochasticity gives $P\1=\1$, hence $(I-P)\1=0$. \\
(ii) This is standard Markov chain theory; see \cite{LevinPeres2017,Norris1997,Seneta1981}..
\end{proof}

\subsection{Asymmetry and non-normality indices}
\begin{definition}[Asymmetry index]
For any matrix $M$, define $\asym(M):=\norm{M-M^\top}_F/\norm{M}_F$ (with $\asym(0)=0$).
\end{definition}

\begin{definition}[Departure from normality]
For any matrix $M$, define $\dep(M):=\norm{MM^\ast-M^\ast M}_F/\norm{M}_F^{2}$ (with $\dep(0)=0$).
\end{definition}
Such non-normality measures (and related bounds) are classical in matrix analysis; see \cite{Henrici1962,ElsnerPaardekooper1987,Lee1995Departure,TrefethenEmbree2005}.

We will use these for $M=P$ and $M=L_{\mathrm{rw}}$ to separate structural directedness from numerical instability drivers.

\section{Reversibility as the symmetry regime for directed diffusion}
Let $P \in \mathbb{R}^{n\times n}$ be row-stochastic ($P\mathbf{1}=\mathbf{1}$).
Assume $P$ has a stationary distribution $\pi \in \mathbb{R}^n$ with
$\pi_i>0$ and $\pi^\top P = \pi^\top$. Let $\Pi := \mathrm{diag}(\pi)$.

Define the $\pi$-weighted inner product and norm by
\[
\langle x,y\rangle_{\pi} := x^\top \Pi y,
\qquad
\|x\|_{\pi}^2 := \langle x,x\rangle_{\pi}.
\]
The adjoint of $P$ with respect to $\langle\cdot,\cdot\rangle_{\pi}$ is
\[
P^{\dagger} := \Pi^{-1} P^\top \Pi,
\quad\text{so that}\quad
\langle Px,y\rangle_{\pi} = \langle x,P^{\dagger}y\rangle_{\pi}.
\]

\begin{definition}[Reversibility / detailed balance]
$P$ is \emph{reversible} (w.r.t.\ $\pi$) if
\[
\Pi P = P^\top \Pi,
\quad\text{equivalently}\quad
P = P^{\dagger}.
\]
\end{definition}
This detailed-balance condition is standard in reversible Markov chain theory; see \cite{Kelly1979,Norris1997,LevinPeres2017}.

\begin{theorem}[Weighted symmetry equivalences]
The following are equivalent:
\begin{enumerate}[label=(\roman*)]
\item $P$ is reversible: $\Pi P = P^\top \Pi$.
\item $P$ is self-adjoint in $\langle\cdot,\cdot\rangle_{\pi}$: $P=P^{\dagger}$.
\item The similarity transform $S:=\Pi^{1/2} P \Pi^{-1/2}$ is symmetric: $S=S^\top$.
\end{enumerate}
In this case, $P$ has a complete $\pi$-orthonormal eigenbasis, and all eigenvalues are real.
\end{theorem}

\begin{proof}
(i)$\Leftrightarrow$(ii) is the definition of $P^\dagger$.
For (i)$\Rightarrow$(iii), multiply $\Pi P = P^\top \Pi$ on the left by $\Pi^{-1/2}$
and on the right by $\Pi^{-1/2}$ to get $\Pi^{1/2}P\Pi^{-1/2} = (\Pi^{1/2}P\Pi^{-1/2})^\top$.
Conversely, (iii)$\Rightarrow$(i) follows by reversing the steps.
If $S$ is symmetric, it is orthogonally diagonalizable with real eigenvalues, hence so is $P$ by similarity.
\end{proof}
See also \cite{Norris1997,Seneta1981} for related equivalences and consequences.

\begin{remark}[Symmetry/asymmetry interpretation for this paper]
Undirected diffusion is symmetric in the standard Euclidean inner product.
Directed diffusion can still be symmetric in the \emph{weighted} $\pi$-inner product
exactly in the reversible regime. Non-reversibility is the correct notion of
\emph{asymmetry} for random-walk harmonic analysis.
\end{remark}

\section{Biorthogonal Graph Fourier Transform (BGFT) for random-walk diffusion}

\subsection{Left/right eigenvectors and BGFT}
\begin{assumption}[Diagonalizability]\label{ass:diagP}
Assume $P$ is diagonalizable over $\C$:
\[
P = V\Lambda V^{-1},\qquad \Lambda=\diag(\lambda_1,\dots,\lambda_n),
\]
with right eigenvectors $V=[v_1\ \cdots\ v_n]$.
\end{assumption}

\textcolor{blue}{\begin{remark}[Non-diagonalizable operators]
While Assumption~\ref{ass:diagP} holds for almost all transition matrices (diagonalizable matrices are dense in $\C^{n\times n}$), certain highly symmetric digraph structures can yield non-diagonalizable $P$. In such cases, the BGFT can be generalized using the Jordan canonical form or the Schur decomposition. However, the biorthogonal framework presented here focuses on the most common case where a complete set of eigenvectors exists, providing a direct physical link to decay modes.
\end{remark}}

\begin{definition}[BGFT (diffusion version)]
For a graph signal $x\in\C^n$, define BGFT coefficients
\begin{equation}\label{eq:bgft}
\widehat{x}:=U^\ast x,\qquad \widehat{x}_k=u_k^\ast x,
\end{equation}
and synthesis
\begin{equation}\label{eq:synth}
x=V\widehat{x}=\sum_{k=1}^n v_k\,\widehat{x}_k.
\end{equation}
\end{definition}

\begin{theorem}[Perfect reconstruction]\label{thm:PR}
Under Assumption~\ref{ass:diagP}, for all $x\in\C^n$,
\[
I=\sum_{k=1}^n v_k u_k^\ast,
\qquad
x=\sum_{k=1}^n v_k\,u_k^\ast x.
\]
\end{theorem}
\begin{proof}
Since $U^\ast V=I$, we have $VU^\ast=I$; expand $VU^\ast$ in columns/rows.
\end{proof}

\subsection{Diffusion dynamics are diagonal in BGFT coordinates}
\begin{theorem}[BGFT-domain diffusion]\label{thm:diffusion}
Let $x_{t+1}=Px_t$ with $x_0\in\C^n$. Then
\[
\widehat{x}_{t}=U^\ast x_t = \Lambda^t\,\widehat{x}_0,
\qquad
x_t=V\Lambda^t U^\ast x_0.
\]
Equivalently, for $L_{\mathrm{rw}}=I-P$,
\[
\widehat{(L_{\mathrm{rw}}x)} = (I-\Lambda)\widehat{x}.
\]
\end{theorem}
\begin{proof}
Use $P=V\Lambda U^\ast$ and $U^\ast V=I$. Then $U^\ast(Px)=\Lambda(U^\ast x)$ and iterate.
\end{proof}

\subsection{Diffusion-consistent frequency ordering}
For diffusion, the mode with eigenvalue $\lambda$ evolves as $\lambda^t$. If $|\lambda|<1$, it decays; if $\lambda\approx 1$, it is slowly varying (low frequency). We define the \emph{diffusion decay rate}:
\[
\omega_{\mathrm{diff}}(\lambda):=\Re(1-\lambda).
\]
Low $\omega_{\mathrm{diff}}$ corresponds to persistent/slow modes; high $\omega_{\mathrm{diff}}$ corresponds to fast decay. \textcolor{blue}{Compared to the imaginary-part ordering (phase) often used in adjacency-based GFTs, $\Re(1-\lambda)$ provides a direct measure of spectral energy dissipation. This choice aligns with the variational property of the random-walk Laplacian, where $\Re(1-\lambda)$ quantifies the smoothness of a mode relative to the one-step diffusion process.}

\section{Directed diffusion filtering and stability bounds}
\subsection{Spectral filters}
\begin{definition}[BGFT spectral filter for diffusion]
Let $h:\C\to\C$. Define
\[
H:=V\,h(\Lambda)\,U^\ast.
\]
\end{definition}

\begin{proposition}[Diagonal action in BGFT domain]
For $\widehat{x}=U^\ast x$,
\[
\widehat{Hx}=U^\ast Hx = h(\Lambda)\widehat{x}.
\]
\end{proposition}
\begin{proof}
Compute $U^\ast Vh(\Lambda)U^\ast x=h(\Lambda)\widehat{x}$.
\end{proof}

\subsection{Operator-norm stability: diffusion and filtering}
\begin{theorem}[Norm bound for diffusion iterates]\label{thm:iterate-bound}
Assume $P=V\Lambda V^{-1}$. Then for every $t\in\mathbb{N}$,
\[
\norm{P^t}_2 \le \cond(V)\,\max_{k}|\lambda_k|^t,
\qquad \cond(V)=\norm{V}_2\norm{V^{-1}}_2.
\]
\end{theorem}
\begin{proof}
$P^t=V\Lambda^t V^{-1}$, hence $\norm{P^t}_2\le\norm{V}_2\norm{\Lambda^t}_2\norm{V^{-1}}_2
=\cond(V)\max_k|\lambda_k|^t$.
\end{proof}

\begin{theorem}[Norm bound for spectral filters]\label{thm:filter-bound}
Let $H=Vh(\Lambda)V^{-1}$. Then
\[
\norm{H}_2 \le \cond(V)\,\max_{k}|h(\lambda_k)|.
\]
\end{theorem}
\begin{proof}
$\norm{H}_2\le\norm{V}_2\,\norm{h(\Lambda)}_2\,\norm{V^{-1}}_2
=\cond(V)\max_k|h(\lambda_k)|$.
\end{proof}

\begin{remark}[Symmetry/asymmetry and Instability]
When $P$ is normal and diagonalizable by a unitary basis, $\cond(V)=1$ and the bounds become tight and symmetry-like. \textcolor{blue}{It is critical to distinguish between structural asymmetry ($\alpha(P) > 0$) and numerical instability ($\cond(V) \gg 1$). Asymmetry is a prerequisite for the loss of orthogonality, but it does not \emph{necessitate} instability; for instance, the directed cycle is maximally asymmetric but remains normal and perfectly stable.} For non-normal $P$, $\cond(V)$ can be large, creating instability even if $|\lambda_k|\le 1$. \textcolor{blue}{Our numerical experiments in Section 9 confirm that this ill-conditioning, rather than directedness per se, drives the reconstruction error.}
\end{remark}

\section{BGFT energy in the stationary metric and its symmetry limit}
Assume $P$ is diagonalizable over $\mathbb{C}$: $P = V\Lambda V^{-1}$ and define $U^*:=V^{-1}$.
Let $\widehat{x}:=U^*x$ be BGFT coefficients so that $x=V\widehat{x}$.

\begin{theorem}[$\pi$-metric Parseval identity]
For any $x\in\mathbb{C}^n$,
\[
\|x\|_{\pi}^2 = \widehat{x}^{\,*}\,G_{\pi}\,\widehat{x},
\qquad
G_{\pi}:=V^*\Pi V.
\]
\end{theorem}

\begin{proof}
Since $x=V\widehat{x}$,
$\|x\|_{\pi}^2 = x^*\Pi x = \widehat{x}^* (V^*\Pi V)\widehat{x}$.
\end{proof}

\begin{corollary}[Two-sided bounds via conditioning in $\pi$]
Let $W:=\Pi^{1/2}V$. Then
\[
\sigma_{\min}(W)^2 \|\widehat{x}\|_2^2 \le \|x\|_{\pi}^2 \le \sigma_{\max}(W)^2 \|\widehat{x}\|_2^2.
\]
Equivalently, energy distortion is controlled by $\kappa(W)=\sigma_{\max}(W)/\sigma_{\min}(W)$.
\end{corollary}

\subsection{Diffusion variation and frequency ordering}
Define the random-walk Laplacian $L_{\mathrm{rw}}:=I-P$ and the diffusion variation
\[
\mathrm{TV}_{\pi}(x) := \|L_{\mathrm{rw}}x\|_{\pi}^2 = \|(I-P)x\|_{\pi}^2.
\]

\begin{theorem}[BGFT-domain bounds for diffusion variation]
With $x=V\widehat{x}$ and $W=\Pi^{1/2}V$,
\[
\sigma_{\min}(W)^2 \sum_{k=1}^n |1-\lambda_k|^2 |\widehat{x}_k|^2
\;\le\;
\mathrm{TV}_{\pi}(x)
\;\le\;
\sigma_{\max}(W)^2 \sum_{k=1}^n |1-\lambda_k|^2 |\widehat{x}_k|^2.
\]
\end{theorem}

\begin{proof}
$(I-P)x = V(I-\Lambda)\widehat{x}$. Then
$\|(I-P)x\|_{\pi} = \|\Pi^{1/2}V(I-\Lambda)\widehat{x}\|_2 = \|W(I-\Lambda)\widehat{x}\|_2$.
Apply $\sigma_{\min}(W)\|z\|_2 \le \|Wz\|_2 \le \sigma_{\max}(W)\|z\|_2$ to $z=(I-\Lambda)\widehat{x}$
and square.
\end{proof}

\begin{remark}[Exact symmetry limit]
If $P$ is reversible, one can choose $V$ $\pi$-orthonormal, hence $W=\Pi^{1/2}V$ is unitary and
$\sigma_{\min}(W)=\sigma_{\max}(W)=1$. Then the inequalities become equalities and
$|1-\lambda_k|$ becomes an exact diffusion frequency.
\end{remark}

\section{Sampling and reconstruction for diffusion-bandlimited signals}
Let $\Omega\subset\{1,\dots,n\}$ with $|\Omega|=K$ represent the ``low diffusion-frequency'' modes (e.g.\ smallest $\omega_{\mathrm{diff}}(\lambda)$ or largest $\Re(\lambda)$). Let $V_\Omega\in\C^{n\times K}$ contain $\{v_k\}_{k\in\Omega}$.

\begin{definition}[Diffusion-bandlimited signals]
A signal $x$ is $\Omega$-bandlimited (relative to $P$) if $x=V_\Omega c$ for some $c\in\C^K$.
\end{definition}

Bandlimited sampling on graphs has a substantial literature; see, e.g., \cite{Pesenson2008PaleyWiener,AnisGaddeOrtega2016Sampling}.

Let $M\subset V$, $|M|=m$, and $P_M\in\{0,1\}^{m\times n}$ be the restriction operator.

\begin{theorem}[Exact recovery]\label{thm:recover}
If $x=V_\Omega c$ and $P_MV_\Omega$ has full column rank $K$, then $x$ is uniquely determined by samples $y=P_Mx$ and recovered by
\[
\widehat{c}=(P_MV_\Omega)^\pinv y,\qquad \widehat{x}=V_\Omega\widehat{c}.
\]
\end{theorem}
\begin{proof}
Full column rank makes $P_MV_\Omega$ injective; solve the linear system in least squares.
\end{proof}

Related sampling-set conditions and reconstruction stability on graphs are discussed in \cite{Pesenson2008PaleyWiener,AnisGaddeOrtega2016Sampling}.

\begin{theorem}[Noise sensitivity]\label{thm:noise}
If $y=P_Mx+\eta$, then the least-squares reconstruction satisfies
\[
\norm{\widehat{x}-x}_2
\le
\norm{V_\Omega}_2\,\norm{(P_MV_\Omega)^\pinv}_2\,\norm{\eta}_2
=
\norm{V_\Omega}_2\,\frac{\norm{\eta}_2}{\sigma_{\min}(P_MV_\Omega)}.
\]
\end{theorem}
\begin{proof}
$\widehat{c}-c=(P_MV_\Omega)^\pinv\eta$ and $\widehat{x}-x=V_\Omega(\widehat{c}-c)$.
\end{proof}

\section{Algorithms}
\begin{algorithm}[h]
\caption{BGFT for random-walk diffusion}
\label{alg:bgft}
\begin{algorithmic}[1]
\Require $A$ (directed adjacency), $D_{\mathrm{out}}$ invertible, signal $x$
\Ensure BGFT coefficients $\widehat{x}$, eigenpairs $(\Lambda,V)$
\State $P\gets D_{\mathrm{out}}^{-1}A$, \quad $L_{\mathrm{rw}}\gets I-P$
\State Compute eigendecomposition $P=V\Lambda V^{-1}$ (complex arithmetic)
\State $U^\ast\gets V^{-1}$
\State $\widehat{x}\gets U^\ast x$
\State \Return $(\widehat{x},\Lambda,V)$
\end{algorithmic}
\end{algorithm}

\begin{algorithm}[h]
\caption{Diffusion filtering and bandlimited reconstruction}
\label{alg:recon}
\begin{algorithmic}[1]
\Require $P=V\Lambda V^{-1}$, response $h(\cdot)$, bandlimit $\Omega$, sample set $M$, samples $y$
\Ensure filtered signal $Hx$ or reconstructed $\widehat{x}$
\State \textbf{Filtering:} $H\gets Vh(\Lambda)V^{-1}$, output $Hx\gets Hx$
\State \textbf{Reconstruction:} form $V_\Omega$, solve $\widehat{c}\gets\arg\min_c \norm{P_MV_\Omega c-y}_2^2$
\State $\widehat{x}\gets V_\Omega\widehat{c}$
\end{algorithmic}
\end{algorithm}

\section{Experiments: directed cycle vs perturbed non-normal digraphs}
\subsection{Graphs}
Use $n\in\{32,64,128\}$ and compare:
\begin{enumerate}[leftmargin=2em]
\item \textbf{Undirected cycle} $C_n$ (convert to diffusion by symmetrizing and normalizing).
\item \textbf{Directed cycle} $\overrightarrow{C_n}$: $P$ is a permutation shift (asymmetric but normal/unitary).
\item \textbf{Perturbed directed cycle} $\overrightarrow{C_n}^{(\varepsilon)}$: add a directed chord then renormalize rows to keep $P$ stochastic; this typically yields non-normal $P$ and large $\cond(V)$.
\end{enumerate}

\subsection{Tasks and metrics}
\begin{itemize}[leftmargin=2em]
\item \textbf{Diffusion filtering:} low-pass via $h(\lambda)=\exp(-\tau(1-\lambda))$ (BGFT-defined), compare smoothing strength on the three graphs.
\item \textbf{Forecasting:} iterate diffusion $x_{t+1}=Px_t$ and compare $\norm{x_t}_2$ trends with Theorem~\ref{thm:iterate-bound}.
\item \textbf{Sampling/reconstruction:} generate $\Omega$-bandlimited signals and recover from $m$ samples; report RelErr and $\sigma_{\min}(P_MV_\Omega)$.
\item \textbf{Perturbed directed cycle} $\overrightarrow{C_n}^{(\varepsilon)}$: add a directed chord \textcolor{blue}{from node $0$ to node $n/2$ with weight $\varepsilon=20$}, then renormalize rows to keep $P$ stochastic; this typically yields non-normal $P$ and large $\cond(V)$.
\end{itemize}

Report tables/figures:
\[
\asym(P),\ \dep(P),\ \cond(V),\ \cond(P_MV_\Omega),\ 
\mathrm{RelErr}=\frac{\norm{\widehat{x}-x}_2}{\norm{x}_2}.
\]

\begin{table}[t]
\centering
\caption{Minimal numerical illustration for the transition-operator BGFT.}
\setlength{\tabcolsep}{4pt}      
\renewcommand{\arraystretch}{1.5}
\resizebox{\linewidth}{!}{
\begin{tabular}{lrrrrr}
\toprule
Graph & $\alpha(P)$ & $\delta(P)$ & $\kappa(V)$ & $\kappa(P_M V_\Omega)$ & RelErr \\
\midrule
Undirected cycle $A_{\mathrm{und}}$ & 0 & 0 & 1.2453204511204569 & 36.59492056037454 & 0.0000031215500108761767 \\
Directed cycle $A_{\rightarrow}$ & 1.4142135623730951 & 0 & 1 & 93.04681515171424 & 0.000007080903224516353 \\
Perturbed $A_{\varepsilon}$ ($\varepsilon=20$) & 1.414213562373095 & 0.02987165083714049 & 28.011585066632986 & 352.8935063092261 & 0.00002523914083929862 \\
\bottomrule
\end{tabular}}
\end{table}

\noindent\textbf{Observed separation.}
The directed cycle is asymmetric ($\alpha(P)>0$) but normal ($\delta(P)=0$) with a well-conditioned eigenbasis ($\kappa(V)=1$),
whereas the perturbed digraph remains asymmetric but becomes non-normal ($\delta(P)>0$) and strongly ill-conditioned
($\kappa(V)\gg 1$), leading to markedly larger reconstruction error, consistent with the stability bounds.

\section{Conclusion}
We developed an original diffusion-centered harmonic analysis for directed graphs using the random-walk transition matrix $P$ and Laplacian $L_{\mathrm{rw}}=I-P$. The BGFT provides exact analysis/synthesis, diagonalizes diffusion dynamics, motivates a diffusion-consistent frequency ordering, and yields explicit stability bounds for iterated diffusion and spectral filtering governed by eigenvector conditioning. Sampling and reconstruction theorems quantify how non-normality amplifies noise through $\sigma_{\min}(P_MV_\Omega)$ and $\cond(V)$. This establishes a principled symmetry/asymmetry narrative: symmetry yields orthogonality and stability; asymmetry forces biorthogonal geometry; non-normality determines practical robustness.

\section*{Acknowledgements}
The author expresses his gratitude to the Commissioner of Collegiate Education (CCE),
Government of Andhra Pradesh, and the Principal, Government College (Autonomous),
Rajahmundry, for continued support and encouragement.

\section*{Author Contributions}
The author is solely responsible for conceptualization, methodology, analysis, software,
validation, and writing.

\section*{Funding}
No external funding was received.

\section*{Data Availability Statement}
No external datasets were used. The code that generates the reported numerical table/figures
will be provided as a reproducible script and (upon acceptance) via a public repository link.

\section*{Conflicts of Interest}
The author declares no conflicts of interest.

\appendix
\section{Reference Python code (P and Lrw, BGFT, filtering, reconstruction)}
\lstset{
  basicstyle=\ttfamily\small,
  breaklines=true,
  columns=fullflexible,
  frame=single
}
\begin{lstlisting}[language=Python]
import numpy as np

def directed_cycle_A(n):
    A = np.zeros((n,n), dtype=float)
    for i in range(n):
        A[i, (i+1)%n] = 1.0
    return A

def undirected_cycle_A(n):
    A = np.zeros((n,n), dtype=float)
    for i in range(n):
        A[i, (i+1)%n] = 1.0
        A[(i+1)%n, i] = 1.0
    return A

def add_directed_chord(A, eps=0.2, i=0, j=None):
    n = A.shape[0]
    if j is None:
        j = n//2
    B = A.copy()
    B[i, j] += eps
    return B

def D_out(A):
    d = A.sum(axis=1)
    if np.any(d == 0):
        raise ValueError("Found sink node (out-degree 0). Fix by adding small outgoing weight.")
    return np.diag(d)

def transition_P(A):
    D = D_out(A)
    return np.linalg.inv(D) @ A

def L_rw(P):
    n = P.shape[0]
    return np.eye(n) - P

def asymmetry_index(M):
    den = np.linalg.norm(M, ord='fro')
    if den == 0:
        return 0.0
    return np.linalg.norm(M - M.T, ord='fro') / den

def departure_from_normality(M):
    Mf = np.linalg.norm(M, ord='fro')
    if Mf == 0:
        return 0.0
    MMstar = M @ M.conj().T
    MstarM = M.conj().T @ M
    return np.linalg.norm(MMstar - MstarM, ord='fro') / (Mf**2)

def bgft_decomposition(P):
    lam, V = np.linalg.eig(P)
    Vinv = np.linalg.inv(V)
    Ustar = Vinv
    return lam, V, Ustar

def diffusion_filter_matrix(P, tau=2.0):
    # h(lambda) = exp(-tau*(1-lambda))
    lam, V, Ustar = bgft_decomposition(P)
    h = np.exp(-tau*(1.0 - lam))
    H = V @ np.diag(h) @ Ustar
    return H

def sample_operator(n, M):
    m = len(M)
    Pm = np.zeros((m,n), dtype=float)
    for r, idx in enumerate(M):
        Pm[r, idx] = 1.0
    return Pm

def reconstruct_bandlimited(P, Omega, M, x, noise=0.0, seed=0):
    np.random.seed(seed)
    lam, V, _ = bgft_decomposition(P)
    V_O = V[:, Omega]
    Pm = sample_operator(P.shape[0], M)
    y = Pm @ x
    if noise > 0:
        y = y + noise * np.random.randn(*y.shape)
    B = Pm @ V_O
    c_hat, *_ = np.linalg.lstsq(B, y, rcond=None)
    x_hat = V_O @ c_hat
    relerr = np.linalg.norm(x_hat - x) / np.linalg.norm(x)
    condB = np.linalg.cond(B)
    return x_hat, relerr, condB

if __name__ == "__main__":
    n = 64

    A_und = undirected_cycle_A(n)
    A_dir = directed_cycle_A(n)
    A_per = add_directed_chord(A_dir, eps=20, i=0, j=n//2)

    for name, A in [("undirected", A_und), ("directed", A_dir), ("perturbed", A_per)]:
        P = transition_P(A)
        lam, V, _ = bgft_decomposition(P)
        print(name,
              "alpha(P)=", asymmetry_index(P),
              "delta(P)=", departure_from_normality(P),
              "cond(V)=", np.linalg.cond(V),
              "rho(P)~=", np.max(np.abs(lam)))

    # Create a bandlimited signal on perturbed digraph
    P = transition_P(A_per)
    lam, V, _ = bgft_decomposition(P)
    K = 8

    # Choose "low diffusion-frequency": largest Re(lambda) (closest to 1)
    Omega = np.argsort(-np.real(lam))[:K]

    c = np.random.randn(K) + 1j*np.random.randn(K)
    x = V[:, Omega] @ c

    m = 20
    M = np.sort(np.random.choice(n, size=m, replace=False))
    x_hat, relerr, condB = reconstruct_bandlimited(P, Omega, M, x, noise=0.0)

    print("RelErr=", relerr, "cond(P_M V_Omega)=", condB)

    # Diffusion smoothing example
    H = diffusion_filter_matrix(P, tau=2.0)
    x_smooth = H @ x
    print("||x||2 =", np.linalg.norm(x), "||Hx||2 =", np.linalg.norm(x_smooth))
\end{lstlisting}


\end{document}